\documentclass[11pt,a4paper]{article}

\usepackage{amsmath,amsfonts,amssymb,latexsym,graphics,epsfig,url}
\usepackage{color}
\usepackage{amsthm}
\usepackage{latexsym}
\usepackage{tikz}
\usetikzlibrary{arrows}

\usepackage[english]{babel}
\usepackage{epsfig}
\usepackage{graphicx}

\newcommand{\executeiffilenewer}[3]{%
\ifnum\pdfstrcmp{\pdffilemoddate{#1}}%
{\pdffilemoddate{#2}}>0%
{\immediate\write18{#3}}\fi%
}

\newtheorem{theorem}{Theorem}[section]

\newtheorem{corollary}[theorem]{Corollary}
\newtheorem{proposicio}[theorem]{Proposition}

\DeclareMathOperator{\spec}{sp}

\def\1{\mbox{\boldmath $1$}}

\def\vec0{\mbox{\boldmath $0$}}

\def\A{\mbox{\boldmath $A$}}

\def\J{\mbox{\boldmath $J$}}

\def\G{\Gamma}

\def\Re{\mathbb R}
\def\Z{\mathbb Z}

\begin{document}
\title{Cospectral digraphs from\\
locally line digraphs
}

\author{C. Dalf\'o$^a$, M. A. Fiol$^b$
\\ \\
{\small $^{a,b}$Departament de Matem\`atiques, Universitat Polit\`ecnica de Catalunya} \\
{\small $^b$Barcelona Graduate School of Mathematics} \\
{\small Barcelona, Catalonia} \\
{\small {\tt\{cristina.dalfo,miguel.angel.fiol\}@upc.edu}} \\
}
\date{}
\maketitle

\begin{abstract}
A digraph $\G=(V,E)$ is a line digraph when every pair of vertices $u,v\in V$ have either equal or disjoint in-neighborhoods. When this condition only applies for vertices in a given subset (with at least two elements), we say that $\G$ is a locally line digraph.
In this paper we give a new method to obtain a digraph $\G'$ cospectral with a given locally line digraph $\G$ with diameter $D$, where the diameter $D'$ of $\G'$ is in the interval $[D-1,D+1]$.
In particular, when the method is applied to De Bruijn or Kautz digraphs, we obtain cospectral digraphs with the same algebraic properties that characterize the formers.
\end{abstract}



\section{Preliminaries}

In this section we recall some basic terminology and simple results concerning digraphs and their spectra. For the concepts and/or results not presented here, we refer the reader to some of
the basic textbooks and papers on the subject; for instance, Chartrand and
Lesniak~\cite{cl96} and Diestel~\cite{d10}.

Through this paper, $\G=(V,E)$ denotes a digraph, with set of vertices $V=V(\G)$ and
set of arcs (or directed edges) $E=E(\G)$, that is
strongly connected (namely, every vertex is connected to any other vertex by traversing the arcs in their corresponding direction).
An arc from vertex $u$ to vertex $v$ is denoted by either $(u,v)$ or $u\rightarrow v$. As usual, we call {\em loop} an arc from a vertex to itself, $u\rightarrow u$, and {\em digon} to two opposite arcs joining a pair of vertices, $u\rightleftarrows v$.
The set of vertices adjacent to and from $v\in V$ is denoted by $\G^{-}(v)$ and
$\G^{+}(v)$, respectively. Such vertices are referred to as {\em in-neighbors} and {\em out-neighbors} of $v$,
respectively. Moreover, $\delta^-(v)=|\G^{-}(v)|$ and $\delta^+(v)=|\G^{+}(v)|$ are the \emph{in-degree} and \emph{out-degree} of vertex $v$, and $\G$ is {\em $d$-regular} when
$\delta^+(v)=\delta^-(v)=d$ for any $v\in V$. Similarly, given $U\subset V$, $\G^{-}(U)$ and $\G^{+}(U)$ represent the sets of vertices
adjacent to and from (the vertices of) $U$.
Given two vertex subsets $X,Y\subset V$, the subset of arcs from $X$ to $Y$ is denoted by $e(X,Y)$.

In the line digraph $L\G$ of a digraph $\G$, each vertex represents an arc of $\G$,
$V(L\G)=\{uv: (u,v)\in E(G)\}$, and a vertex $uv$ is adjacent to a vertex $wz$ when $v=w$, that is, when in $\G$ the arc $(u,v)$ is adjacent to the arc $(w,z)$: $u\rightarrow v(=w)\rightarrow z$. By the Heuchenne's condition~\cite{He64}, a digraph $\G$ is a line digraph if and only if, for every pair of vertices $u,v$, either $\G^+(u)=\G^+(v)$ or $\G^+(u)\cap\G^+(v)=\emptyset$.
Since the line digraph of the converse digraph $\overline{\G}$ (obtained from $\G$ by reversing the directions of all the arcs) equals the converse of the line digraph, $L\overline{\G}=\overline{L\G}$, the above condition can be restated in terms of the in-neighborhoods $\G^-(u)$ and $\G^-(v)$. In particular, we say that a digraph is a {\em $($U-$)$locally line digraph} if there is a vertex subset $U$ with at least two elements such that $\G^-(u)=\G^-(v)$ for every $u,v\in U$.

In the case of graphs instead of digraphs, the Godsil-McKay switching given in~\cite{GoMc93} is a technique
to obtain cospectral graphs.

\section{Main result}
The following result describes the basic transformation of a digraph $\G$ into another digraph $\G'$
modifying slightly the walk properties of the former (see Figure~\ref{fig:cjts-X-Y-Z}).

\begin{figure}[t]
    \vskip-.5cm
    \begin{center}
        \includegraphics[width=12cm]{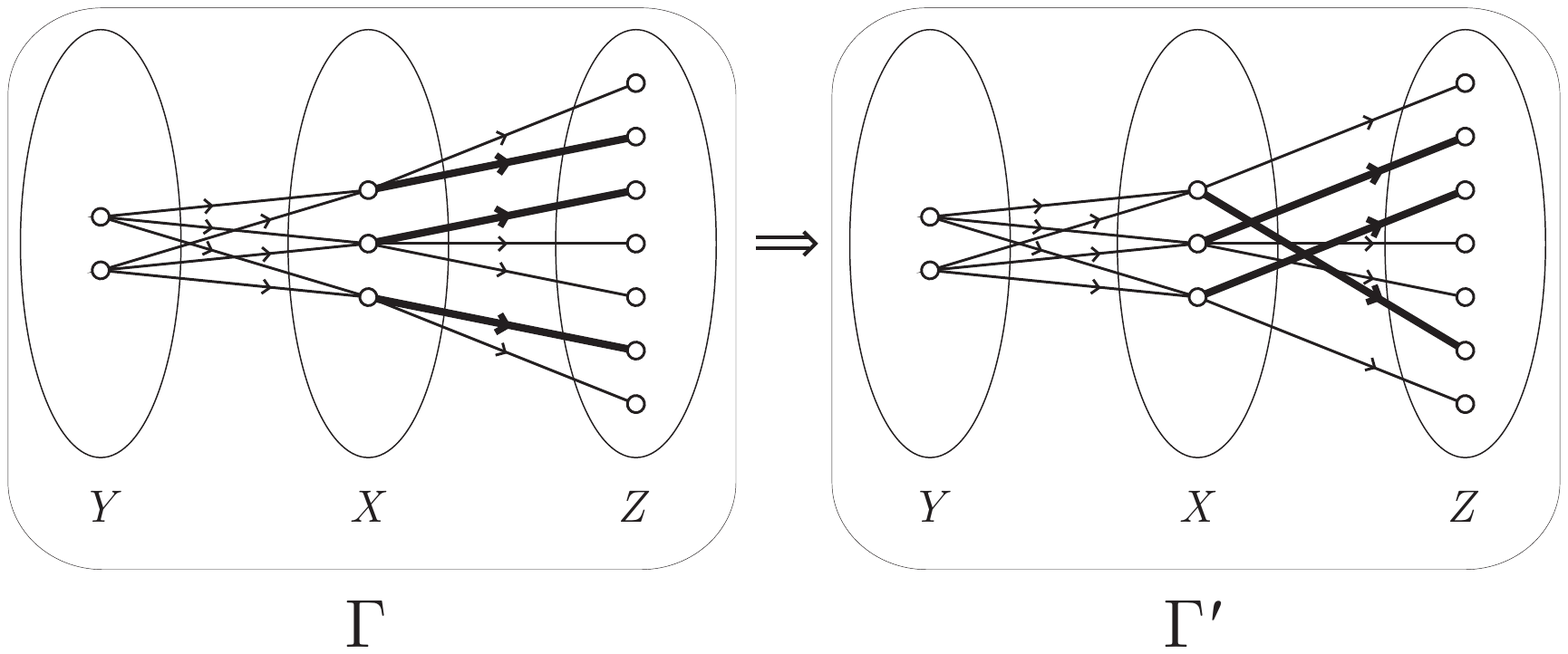}
    \end{center}
    \vskip-3.75cm
	\caption{Scheme of the sets of Theorem~\ref{teo:principal0}. The arcs that change from $\G$ to $\G'$ are represented with a thick line.}
	\label{fig:cjts-X-Y-Z}
\end{figure}

\begin{theorem}
\label{teo:principal0}
Let $\G=(V,E)$ be a digraph with diameter $D\ge 2$.
Consider a subset of vertices $X=\{x_1,\ldots,x_r\}\subset V$, $r\ge 2$,
such that the sets of the in-neighbors  of every $x_i$ are the same for every $x_i$, say, $Y=\G^{-}(x_i)$ for $i=1,\ldots,r$.  Let $Z=\G^{+}(X)$.
Let $\G'$ be the {\em modified digraph} obtained from $\G$ by changing the set of arcs $e(X,Z)$ by another
set of arcs $e'(X,Z)$ in such a way that the two following conditions are satisfied:
\begin{itemize}
\item[$(i)$]
The loops remain unchanged, that is, with $e'(Y,X)$ being the set of arcs from
$Y$ to $X$ in $\G'$, we must have $e'(Y,X)\cap e'(X,Z)=e(Y,X)\cap e(X,Z)$.
\item[$(ii)$]
For the arcs that are not loops, every vertex of $X$ has some out-going arcs to a vertex of $Z$, and every vertex of $Z$ gets some in-going arcs from a vertex of $X$.
\end{itemize}
Assume that there is a walk of length $\ell\ge 2$ from $u$ to $v$ $(u,v\in V)$ in $\G$.
\begin{itemize}
\item[$(a)$]
If $u\not\in X$, then there is also a walk of length $\ell$ from $u$ to $v$ in $\G'$.
\item[$(b)$]
If $u\in X$, then there is a walk of length at most $\ell+1$ from $u$ to $v$ in $\G'$.
\end{itemize}
\end{theorem}

\begin{proof}
$(a)$ Let $u_0(=u),u_1,\ldots, u_{\ell-1}(=v)$ be an $\ell$-walk from $u$ to $v$ in $\G$. We distinguish two
cases:
\begin{enumerate}
\item
If $u_i\notin X$ for every $i=1,\ldots,\ell-2$, the result is trivial as the walk in $\G'$ is the same as
that in $\G$.
\item
If $u_i\in X$ for some $i=1,\ldots,\ell-2$, then from the hypothesis on $X$ we must have $u_{i-1}\in Y$ and
$X\cap\G^+(u_{i-1})=X$. Moreover, by $(ii)$, in $\G'$ there is a vertex $u_i'\in X$ adjacent to $u_{i+1}$.
Thus, the required $\ell$-walk in $\G'$ is just $u_0,\ldots,u_{i-1},u_i',u_{i+1},\ldots,u_{\ell-1}$.
\end{enumerate}
$(b)$ If $u\in X$, the result is a simple consequence of $(a)$. Indeed, by $(ii)$ there is a vertex $u'\in Z\setminus X$ adjacent from $u$ (otherwise, $\G$ would not be strongly connected). Then, it suffices to consider the walk $u,u',\ldots,v$.
This completes the proof.
\end{proof}

If we consider shortest walks, the following consequence is straightforward.
\begin{corollary}
If $\G$ is a digraph with diameter $D$, the modified digraph $\G'$ $($in the sense of Theorem
\ref{teo:principal0}$)$ has diameter $D'$ satisfying $D-1\le D'\le D+1$.
\end{corollary}

Note that the case $D'=D-1$ could happen when, in $\G$, all vertices not in $X$ have eccentricity $D-1$ and, in
$\G'$ all vertices in $X$ result with the same eccentricity $D-1$.

Examples of the case when the diameter remains unchanged, $D'=D$, are provided by the modified De Bruijn digraphs (see Section~\ref{secD}).

\section{Cospectral digraphs}
First notice that, because of the condition $Y=\Gamma^{-}(x_i)$, $i=1,\ldots,r$,
the spectrum of $\G$ contains the eigenvalue $0$ with multiplicity $m(0)\ge r-1$.
Indeed, suppose that its adjacency matrix $\A$ is indexed in such a way that the first $r$ rows correspond to
the vertices of $X$. Then, the $r-1$ (column) vectors
$(1,-1,0,0,0,\ldots,0)$, $(0,1,-1,0,0,\ldots,0)$,\ldots, $(0,\ldots,0,1,-1,0,\ldots,0)$ are clearly linearly
independent, and they are also eigenvectors with eigenvalue $0$. For more details, see Fiol and Mitjana~\cite{fm07}.

Another interesting consequence of Theorem~\ref{teo:principal0} is the following relationship between the adjacency matrices of $\G$ and $\G'$, in the particular case when the in-degrees of the vertices of $Z$ are
preserved.

\begin{proposicio}
\label{propoA'A}
Assume that in the modified digraph $\G'$ from $\G$, every vertex of $Z$ gets the same number of
in-going arcs as in $\G$. That is,
$\G'^{-}(v)=\G^-(v)$ for every $v\in Z$.
Let $\A=(a_{uv})$ and $\A'=(a'_{uv})$ be the adjacency matrices of $\G$ and $\G'$, respectively.
Then, for any polynomial $p\in \Re[x]$ without constant term, say, $p(x)=xq(x)$, with $\deg q=\deg p-1$, we
have
\begin{equation}
\label{matrixEq}
p(\A')=\A'q(\A).
\end{equation}
\end{proposicio}
\begin{proof}
We only need to prove that $\A'\A=\A'\A'$. Since the only modified arcs are those adjacent from the vertices of
$X$, we have
\begin{align*}
(\A'\A)_{uv} &=\sum_{x\in X}a'_{ux}a_{xv}+\sum_{x\notin X}a'_{ux}a_{xv}
 =|X\cap \G^-(v)|+\sum_{x\notin X}a'_{ux}a_{xv}\\
  &=|X\cap \G'^-(v)|+\sum_{x\notin X}a'_{ux}a_{xv}
 =\sum_{x\in X}a'_{ux}a'_{xv}+\sum_{x\notin X}a'_{ux}a'_{xv}=
 (\A'\A')_{uv},
\end{align*}
where we used that every vertex of $Z$ in $\G'$ gets the same number of in-going arcs as in $\G$.
\end{proof}

\begin{proposicio}
\label{propo-cospectral}
Within the conditions of Proposition~\ref{propoA'A}, the digraphs
$\G$ and $\G'$ are cospectral.
\end{proposicio}
\begin{proof}
First, note that Eq.~\eqref{matrixEq} is equivalent to state that, for any polynomial $q\in \Re[x]$,
$$
 \A'q(\A')=\A'q(\A).
$$
In particular, if $q(x)=\phi_{\G}(x)$ is the characteristic polynomial of $\G$, the above equation gives
$$
\A'\phi_{\G}(\A')=\A'\phi_{\G}(\A)=0,
$$
so that the polynomial $x\phi_{\G}(x)$ is a multiple of the characteristic polynomial $\phi_{\G'}(x)$ of $\G'$, say,
$x\phi_{\G}(x)=r(x)\phi_{\G'}(x)$ with $\deg r=1$.
Analogously, since $\G$ can be seen as a modified digraph of $\G'$ ($G$ satisfies Proposition \ref{propoA'A}), we get
$x\phi_{\G'}(x)=s(x)\phi_{\G}(x)$ with $\deg s=1$.
Then, we deduce that $\phi_{\G}(x)$ and $\phi_{\G'}(x)$ can only differ by a constant, but, as they
are both monic polynomials, $\phi_{\G}(x)=\phi_{\G'}(x)$ and $\spec \G=\spec \G'$, as claimed.
\end{proof}

Given a digraph $\G$, its converse digraph $\overline{\G}$ has the same vertex set as $\G$, but all the directions of the arcs are reversed. Then, the walks of $\G$ and $\overline{\G}$ are in correspondence, and, as the adjacency matrix of $\overline{\G}$
is the transpose of that of $\G$, both digraphs are cospectral.
These facts leads us to the symmetric-like result of Theorem \ref{teo:principal0}
and Proposition \ref{propo-cospectral}:

\begin{corollary}
\label{coro:principal0}
Let $\G=(V,E)$ be a digraph with diameter $D\ge 2$.
Consider a subset of vertices $X'=\{x_1,\ldots,x_r\}\subset V$, $r\ge 2$,
such that the sets of the out-neighbors  of every $x_i$ are the same for every $x_i$,
say, $Y'=\G^{+}(x_i)$ for $i=1,\ldots,r$.  Let $Z'=\G^{-}(X')$.
Let $\G'$ be the {\em modified digraph} obtained from $\G$ by changing the set of arcs $e(Z',X')$ by another set of arcs $e'(Z',X')$ in such a way that the two following conditions are satisfied:
\begin{itemize}
\item[$(i)$]
The loops remain unchanged, that is, with $e'(X',Y')$ being the set of arcs from $X'$ to $Y'$
in $\G'$, we must have $e'(X',Y')\cap e'(Z',X')=e(X',Y')\cap e(Z',X')$.
\item[$(ii)$]
For the arcs that are not loops, every vertex of $X'$ has some in-going arcs from a vertex of $Z'$, and every vertex of $Z'$ gets some out-going arcs to a vertex of $X'$.
\end{itemize}
Then, the following hold.
\begin{itemize}
\item[$(a)$]
The diameter $D'$ of $\G'$ lies between $D-1$ and $D+1$.
\item[$(b)$]
If, in the modified digraph $\G'$, every vertex of $Z$ gets the same
out-going arcs as in $\G$, then $\G'$ and $\G$ are cospectral.
\end{itemize}
\end{corollary}
\begin{proof}
Modify the converse digraph of $\G$ according to Theorem \ref{teo:principal0} and Proposition \ref{propo-cospectral}, and then take the converse digraph of the result.
\end{proof}

\section{The modified De Bruijn digraphs}
\label{secD}
The results of the preceding section can be used to obtain digraphs with specific distance-related or walk
properties. Let us begin with the case of the so-called equi-reachable digraphs, of which the well-known De Bruijn digraphs are a particular example.

Let $\G=(V,E)$ be a digraph with diameter $D$, and suppose that, for every pair of vertices $u,v\in V$, there is a walk of constant length $m(\ge D)$ from $u$ to $v$. If $\ell$ is the smallest of such an $m$, we say that $\G$ is {\em $\ell$-reachable}. Some times the term {\em equi-reachable} is used for digraphs that are $\ell$-reachable (for some $\ell$), that is, for digraphs with walks of equal length between vertices.

If $\Gamma$ is $\ell$-reachable and has maximum out-degree $d$, then its order is at most $N = d^{\ell}$, since this is the maximum number of different walks of length $\ell$ from a given vertex. To attain this bound there should be
just one walk of length $\ell$ between any two vertices. Then,
the adjacency matrix $\A$ of $\Gamma$ must verify the matrix equation
 \begin{equation}
\label{matrixEqdB}
\A^{\ell}=\J,
\end{equation}
and, therefore, $\Gamma$ must be $d$-regular,
see Hoffman and McAndrew~\cite{hmc65}. Note
also that these digraphs must be geodesic (that is, with just one
shortest path between any two vertices).

The $\ell$-reachable digraphs with $d^{\ell}$ vertices were studied by Mendelsohn in~\cite{m70} as {\em UPP
digraphs} (digraphs with the unique
path property of order $\ell$), and by Conway and Guy~\cite{cg82}, unaware of the work of Mendelsohn, as
{\em tight precisely $\ell$-steps digraphs},
using them to construct large transitive digraphs of given diameter.
Equi-reachable digraphs were also studied by Fiol, Alegre, Yebra, and F\`abrega~\cite{FiAlYeFa85}.

Among the UPP digraphs, there are the well-known De Bruijn (or Good-De Bruijn) digraphs $B(d,\ell)$, whose set of
vertices consists of all words of length $\ell$ from an alphabet of $d$ symbols, say $\Z_d=\{0,1,\ldots,d-1\}$, and a vertex $x$
is adjacent to a vertex $y$ if the last $\ell-1$ symbols of $x$ coincide with the first $\ell-1$ symbols of $y$.
The De Bruijn digraphs $B(2,\ell)$ for $\ell=1,2,3,4$ are shown in
Figure~\ref{fig:de-bruijn-normals}. In general, it is well-known that the digraph $B(d,\ell)$ is $d$-regular
with diameter $D=\ell$, and it is the line digraph of $B(d,\ell-1)$. Moreover, its adjacency matrix satisfies Eq.~\eqref{matrixEqdB}, which,
as said before, it is the algebraic condition for being $\ell$-reachable.
For more details, see Fiol, Yebra and Alegre~\cite{FiYeAl83,FiYeAl84}.

\begin{figure}[t]
    \vskip-.5cm
    \begin{center}
        \includegraphics[width=15cm]{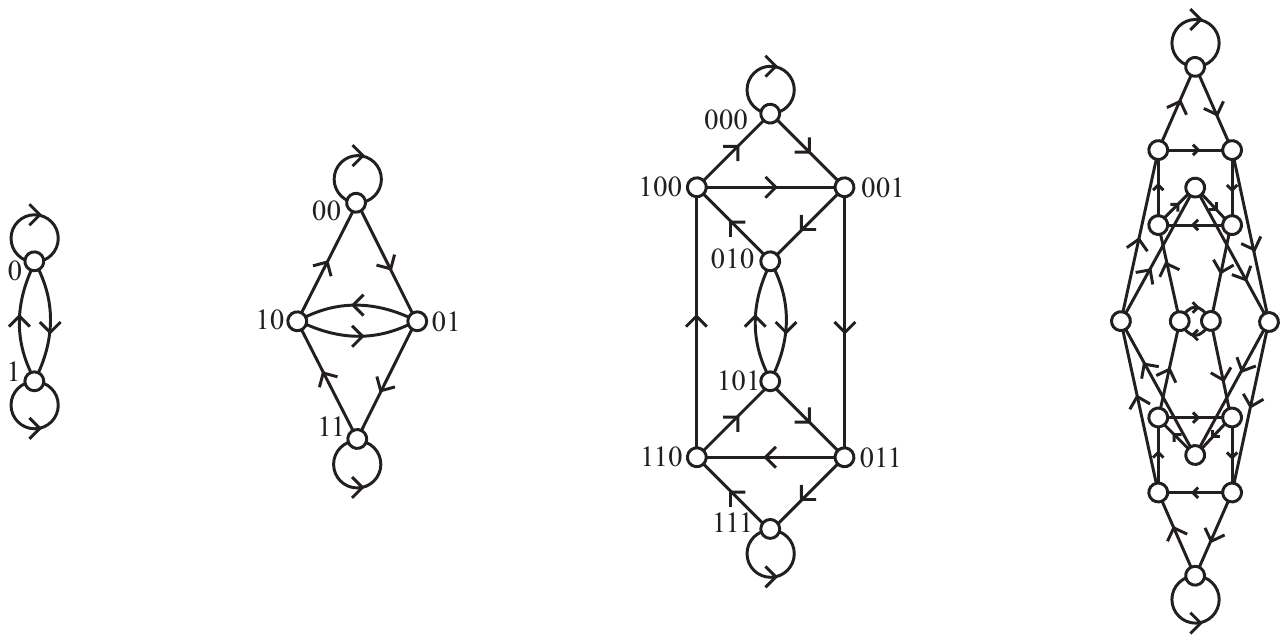}
    \end{center}
    \vskip-5.5cm
	\caption{The De Bruijn digraphs $B(2,1)$, $B(2,2)$, $B(2,3)$, and $B(2,4)$.}
	\label{fig:de-bruijn-normals}
\end{figure}

The De Bruijn digraphs are not the only UPP digraphs. For instance, for $d=3$ and
$\ell=2$ Mendelsohn presented in~\cite{m70} five other nonisomorphic such digraphs that can be seen as models
of groupoids. More generally,
UPP digraphs can be seen as models of a universal algebra, for more information see
Mendelsohn~\cite{m68}.

To obtain a UPP digraph by modifying  $B(d,\ell)$ according to Proposition~\ref{propoA'A}, we need the
modified digraph $B'(d,\ell)$ to have the same diameter $\ell$, as shown in the following result.

\begin{proposicio}
\label{sequence-dB-modify}
Let $\G=B(d,\ell)$. For some fixed values $x_i\in \Z_{d}$, $i=1,2,\ldots,\ell-1$, not all of them being equal (to avoid loops), consider the vertex set
$X=\{x_1x_2\ldots x_{\ell-1}k : k\in \Z_{d}\}$.
Let $\alpha_j$, for $j\in \Z_{d}$, be $d$-permutations of $0,1,\ldots,d-1$.
Let $\G'=B'(d,\ell)$ be the modified digraph obtained by changing the out-going arcs of $X$ in such a way that
every vertex $x_1x_2\ldots x_{\ell-1}k\in X$ is adjacent to the $d$ vertices
\begin{equation}
\label{vertex-of-Z}
x_2x_3\ldots x_{\ell-1}\alpha_j(k)j,\qquad k=0,1,\ldots,d-1.
\end{equation}
Then $\G'$ is a $d$-regular digraph with the same diameter $D'=\ell$ as $\G=B(d,\ell)$, and it is $\ell$-reachable.
\end{proposicio}
\begin{proof}
First, we only need to prove in-regularity (that is, constant in-degree) for every vertex of $Z=\G^+(X)$ given by
\eqref{vertex-of-Z}. But such a vertex is adjacent from the vertices
$$
hx_2\ldots x_{\ell-1}\alpha_j(k),\quad h\neq x_1,\quad\mbox{and}\quad x_1x_2\ldots x_{\ell-1}k.
$$
Moreover, according to Theorem~\ref{teo:principal0}, it suffices to show that, from each vertex
$u=x_1x_2\ldots x_{\ell-1}k\in X$, there is an
$\ell$-walk from $u$ to every other vertex $v=z_1z_2\ldots z_{\ell-1}z_{\ell}$ in $\G'$. To this end, we consider the following walk $u_0(=u),u_1,\ldots,u_{\ell-1},u_{\ell}$ with
\begin{align*}
u_0 &=x_1x_2\ldots x_{\ell-1}k,\\
u_1 &= x_2x_3x_4\ldots x_{\ell-1}\alpha_{y_1}(k)y_1,\\
u_2 &=x_3x_4\ldots x_{\ell-1}\alpha_{y_1}(k)\alpha_{y_2}(y_1)y_2,\\
u_3 &= x_4\ldots x_{\ell-1}\alpha_{y_1}(k)\alpha_{y_2}(y_1)\alpha_{y_3}(y_2)y_3,\\
    & \hskip.2cm\vdots\\
u_{\ell} &= \alpha_{y_2}(y_1)\alpha_{y_3}(y_2)\ldots\alpha_{y_{\ell}}(y_{\ell-1})y_{\ell},
\end{align*}
where, if $u_i\not\in X$ for some $i$, it is assumed that in $u_{i+1}$ all the $\alpha_j$'s are the identity (since there are no changes in the out-going arcs of the former), and
$$
y_{\ell}=z_{\ell},\quad y_{\ell-1}=\alpha_{y_{\ell}}^{-1}(z_{\ell-1}),\ \ldots\ ,
y_2=\alpha_{y_3}^{-1}(z_2),\quad y_1=\alpha_{y_2}^{-1}(z_1),
$$
so giving $u_{\ell}=z_1z_2\ldots z_{\ell-1}z_{\ell}=v$, as desired.
\end{proof}

By way of example,
consider the modified De Bruijn digraph of Figure~\ref{fig:de-bruijn-estrany}, obtained
from $B(2,3)$ by considering the set $X=\{100,101\}$ (so that $Y=\{010,110\}$), and removing the arcs
$100\rightarrow 001$ and $101\rightarrow 011$ 
to set $100\rightarrow011$ and $101\rightarrow 001$. (This corresponds to take the permutations $\alpha_0=\iota$ (the identity) and $\alpha_1=(01)$). Such a digraph was first shown by Fiol, Alegre, Yebra, and F\`abrega in~\cite{FiAlYeFa85}.


\begin{figure}[t]
    \begin{center}
    \hskip-1cm
        \includegraphics[width=16cm]{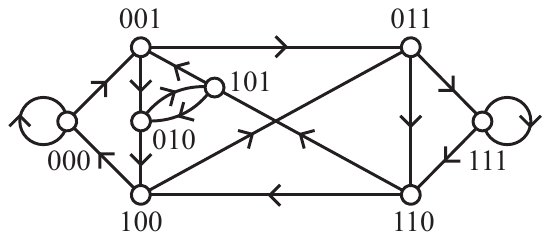}
    \end{center}
    \vskip-19.75cm
	\caption{The modified De Bruijn digraph $B'(2,3)$.}
	\label{fig:de-bruijn-estrany}
\end{figure}

The adjacency matrices of $B(2,3)$ and $B'(2,3)$, with the modified 1's in bold, are, respectively,

\begin{figure}[t]
    \begin{center}
    \hskip-1cm
        \includegraphics[width=18cm]{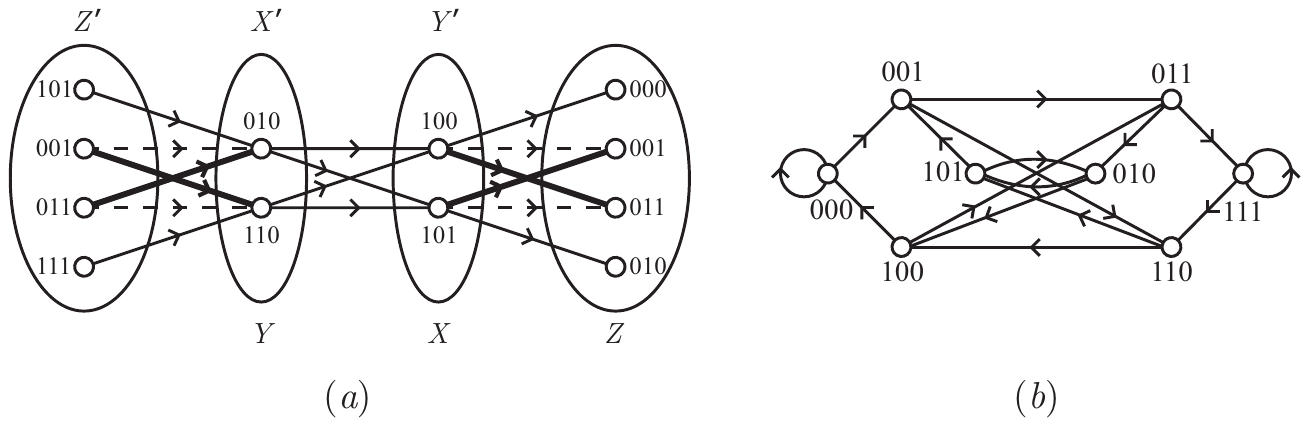}
    \end{center}
    \vskip-8.75cm
	\caption{$(a)$ The ``left and right modifications'' of $B^*(2,3)$; $(b)$ The doubly modified De Bruijn digraph $B^*(2,3)$.}
	\label{fig:doble-bruijn-estrany}
\end{figure}

$$
\A=\left(
  \begin{array}{cccccccc}
    1 & 1 & 0 & 0 & 0 & 0 & 0 & 0 \\
    0 & 0 & 1 & 1 & 0 & 0 & 0 & 0 \\
    0 & 0 & 0 & 0 & 1 & 1 & 0 & 0 \\
    0 & 0 & 0 & 0 & 0 & 0 & 1 & 1 \\
    1 & \1 & 0 & 0 & 0 & 0 & 0 & 0 \\
    0 & 0 & 1 & \1 & 0 & 0 & 0 & 0 \\
    0 & 0 & 0 & 0 & 1 & 1 & 0 & 0 \\
    0 & 0 & 0 & 0 & 0 & 0 & 1 & 1 \\
  \end{array}
\right)\quad
\mbox{ and }
\quad
\A'=\left(
  \begin{array}{cccccccc}
    1 & 1 & 0 & 0 & 0 & 0 & 0 & 0 \\
    0 & 0 & 1 & 1 & 0 & 0 & 0 & 0 \\
    0 & 0 & 0 & 0 & 1 & 1 & 0 & 0 \\
    0 & 0 & 0 & 0 & 0 & 0 & 1 & 1 \\
    1 & 0 & 0 & \1 & 0 & 0 & 0 & 0 \\
    0 & \1 & 1 & 0 & 0 & 0 & 0 & 0 \\
    0 & 0 & 0 & 0 & 1 & 1 & 0 & 0 \\
    0 & 0 & 0 & 0 & 0 & 0 & 1 & 1 \\
  \end{array}
\right).
$$
According to Proposition~\ref{propo-cospectral}, both digraphs are cospectral with
$$
\spec B(2,3)=\spec B'(2,3)=\{0^7,2^1\},
$$
where the superscripts denote the (algebraic)
multiplicites of the eigenvalues.

Observe that $B(2,3)$ and $B'(2,3)$, shown in Figs. \ref{fig:de-bruijn-normals} and \ref{fig:de-bruijn-estrany} respectively, are not isomorphic since, for instance, the former has two cycles (closed walks without repeated vertices) of length $5$:
\begin{align*}
 &000\rightarrow001\rightarrow011\rightarrow110\rightarrow100\rightarrow000,
\quad\mbox{and}\quad
111\rightarrow110\rightarrow100\rightarrow001\rightarrow011\rightarrow111,
\end{align*}
whereas the latter  has three:
\begin{align*}
 &000\rightarrow001\rightarrow011\rightarrow110\rightarrow100\rightarrow000,
\quad
111\rightarrow110\rightarrow101\rightarrow001\rightarrow011\rightarrow111,\\
\quad \mbox{and}\quad
& 010\rightarrow100\rightarrow011\rightarrow110\rightarrow101\rightarrow010.
\end{align*}

Of course, the above situation is not the general case.  Many digraphs obtained by using the modifications described in Proposition \ref{sequence-dB-modify} are cospectral, but also isomorphic to the original digraph.
So, an interesting open problem would be to determine the conditions on the $d$-permutations $\alpha_j$, for $j=0,\ldots,d-1$, to obtain nonisomorphic cospectral digraphs.

In fact, a computer exploration shows that the only nonisomorphic $3$-reachable $2$-regular digraphs are $B(2,3)$, $B'(2,3)$, and $B''(2,3)=\overline{B'(2,3)}$, the converse digraph of $B'(2,3)$, which can be also obtained by using our method.
Indeed, it suffices to take $\overline{B(2,3)}$ and apply Corollary \ref{coro:principal0}
with $X'=X=\{100,101\}$ (so that $Y'=Y\{010,110\}$), and change the same arcs as before, but now with opposite directions.

Another possible interesting perturbation is to apply a double modification:
The one proposed in Theorem \ref{teo:principal0} (or Proposition \ref{propo-cospectral}) with the sets $Y$, $X$, and $Z(=\G^+(X))$; and that of Corollary \ref{coro:principal0} with $Z'=\G^-(X')$, $X'(=Y)$, and $Y'(=X)$. For example, in the case of $B(2,3)$, these modifications are depicted in Fig. \ref{fig:doble-bruijn-estrany}$(a)$, where the dashed arcs are changed to the bold ones, and the obtained digraph is shown in Fig. \ref{fig:doble-bruijn-estrany}$(b)$. Notice that, in this case, we are not longer under the conditions of Proposition \ref{sequence-dB-modify} and, hence, the resulting digraph $B^*(2,3)$ with adjacency matrix $\A$, although still cospectral with  $B(2,3)$, is not a UPP digraph, that is, $\A^3\neq \J$. (But, in fact, we have $\A^4=2\J$, which indicates the existence of exactly 2 walks of length 4 between any two vertices.)

\section{The modified Kautz digraphs}

\begin{figure}[t]
    \vskip-.5cm
    \begin{center}
        \includegraphics[width=15cm]{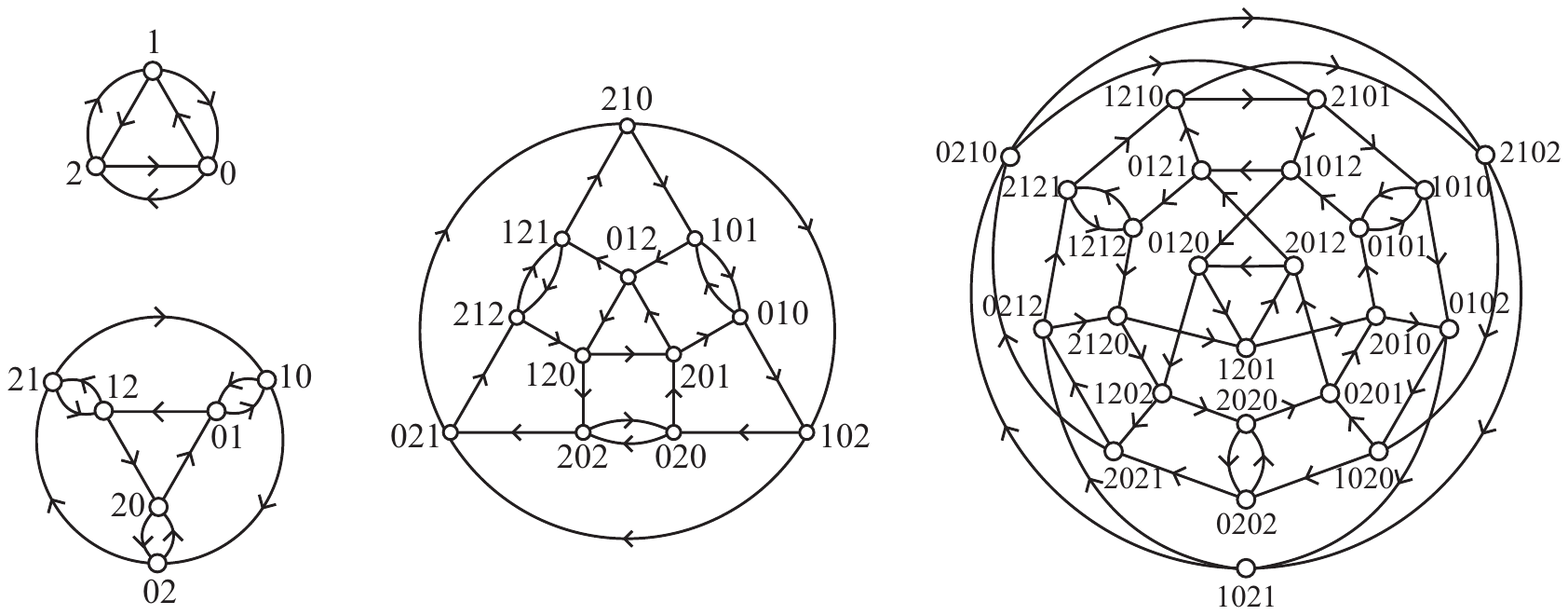}
    \end{center}
    \vskip-5.35cm
	\caption{The Kautz digraphs $K(2,1)$, $K(2,2)$, $K(2,3)$, and $K(2,4)$.}
	\label{fig:kautz-normals}
\end{figure}

\begin{figure}[t]
    \begin{center}
        \includegraphics[width=12cm]{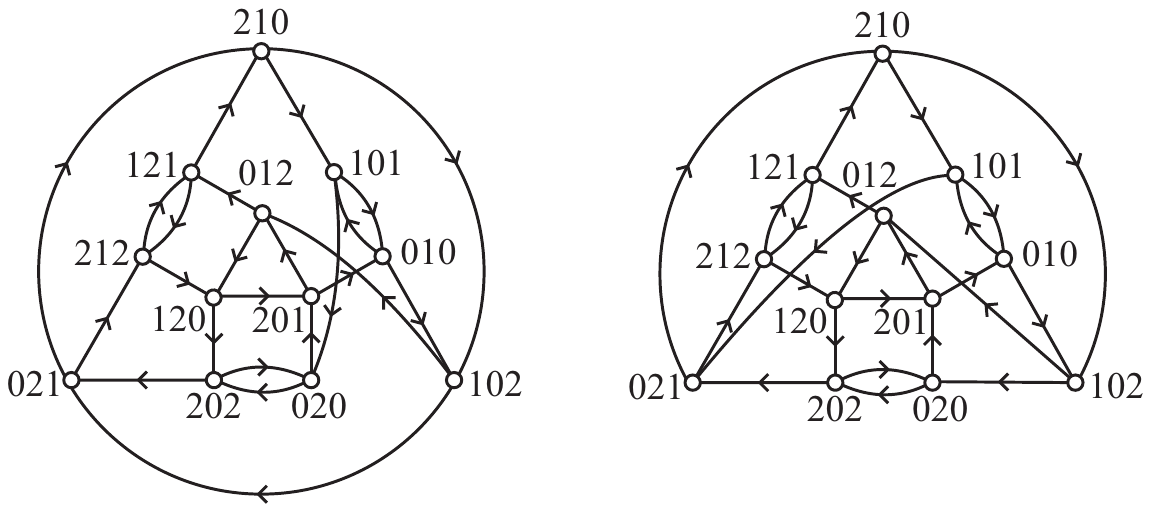}
    \end{center}
    \vskip-12.15cm
	\caption{The modified Kautz digraphs $K'(2,3)$ and $K''(2,3)$.}
	\label{fig:kautz-estrany}
\end{figure}

The Kautz digraph $K(d,\ell)$ is defined as the De Bruin digraph $B(d,\ell)$ but now the consecutive symbols $x_i$ and $x_{i+1}$, taken from the alphabet $\{0,1,\ldots,d\}$, must be different.
The first four Kautz digraphs $K(2,\ell)$, for $\ell=1,2,3,4$, are represented in Figure~\ref{fig:kautz-normals}. Again, it is well-known that any
of these digraphs is the line digraph of the previous one (see Fiol, Yebra, and Alegre \cite{FiYeAl84}).
The adjacency matrix $\A$ of the Kautz digraph $K(d,\ell)$ satisfies the matrix equation
\begin{equation}
\label{matrixEqKa}
\A^{\ell}+\A^{\ell-1}=\J,
\end{equation}
so that between every pair of vertices $u,v$ there is exactly one walk of
 length $\ell$ or $\ell-1$.

Contrarily to the De Bruijn digraphs, some experimental results seems to show that all the modified Kautz digraphs $K'(d,\ell)$ have diameter $D'=\ell+1$.
For example, Figure~\ref{fig:kautz-estrany} shows two modified Kautz digraphs, $K'(2,3)$ and $K''(2,3)$, where,
in both cases, $X=\{101,102\}$ (so that $Y=\{010,210\}$). Then, $K'(2,3)$ is obtained by removing the arcs $101\rightarrow012$ and $102\rightarrow020$ to set the new arcs $101\rightarrow020$ and $102\rightarrow012$; whereas $K''(2,3)$
is obtained by changing $101\rightarrow012$ and $102\rightarrow021$ to get $101\rightarrow021$ and $102\rightarrow012$.

In concordance with Proposition~\ref{propo-cospectral}, all these digraph are cospectral with
$$
\spec K(2,3) = \spec K'(2,3) = \spec K''(2,3)=\{-1^2,0^9,2^1\}.
$$

\textbf{Acknowledgments. } This research is supported by the
{\em Ministerio de Ciencia e Innovaci\'on} and the {\em European Regional
Development Fund} under project MTM2014-60127-P, and the {\em Catalan Research
Council} under project 2014SGR1147.



\begin{thebibliography}{99}







\bibitem{cl96}
G. Chartrand and L. Lesniak, {\em Graphs \& Digraphs}, third ed., Chapman and Hall, London, 1996.

\bibitem{cg82}
J. H. Conway and M. J. T. Guy, Message graphs, {\em Annals of Discrete
Mathematics}, {\bf 13} (Proc. of the Conf, on Graph Theory. Cambridge,
1981), North Holland, 1982, 61--64.


\bibitem{d10}
R. Diestel,  {\em Graph Theory} (4th ed.), Graduate Texts in Mathematics {\bf 173}, Springer-Verlag,
Heilderberg, 2010.

\bibitem{FiAlYeFa85}
M A. Fiol, I. Alegre, J. L. A. Yebra and J. F\`abrega,
Digraphs with walks of equal length between vertices, in:
{\em Graph Theory and its Applications to Algorithms and Computer Science}, Eds. Y. Alavi et al., pp.
313--322, John Wiley, New York, 1985.

\bibitem{fm07}
M. A. Fiol and M. Mitjana,
The spectra of some families of digraphs, \emph{Linear Algebra Appl.} {\bf 423} (2007) 109--118.

\bibitem{FiYeAl83}
M. A. Fiol, J. L. A. Yebra and I. Alegre,
Line digraph iterations and the $(d,k)$  problem for directed graphs,
{\em Proc. 10th Int. Symp. Comput. Arch.,} Stockholm (1983) 174-177.

\bibitem{FiYeAl84}
M. A. Fiol, J. L. A. Yebra and I. Alegre,
Line digraph iterations and the $(d,k)$ digraph problem,
{\em IEEE Trans. Comput.}, {\bf C-33} (1984) 400-403.


\bibitem{GoMc93}
C. D. Godsil and B. D. McKay,
Constructing cospectral graphs, \emph{Aequationes Math.} \textbf{25} (1982) 257--268.

\bibitem{He64}
C. Heuchenne, Sur une certaine correspondance entre graphes,
{\em Bull. Soc. Roy. Sci. Liège} {\bf 33} (1964) 743–-753.



\bibitem{hmc65}
A. J. Hofmann and M. H. McAndrew, The polynomial of a directed
graph, {\em Proc. Amer. Math. Soc.} {\bf 16} (1965) 30--309.

\bibitem{m68}
N. S. Mendelsohn, An application of matrix theory to a problem
in universal algebra, {\em Lineat Algebra Appl.} {\bf  1}
(1968) 471--478.

\bibitem{m70}
N. S. Mendelsohn, Directed graph with the unique path property,
{\em Combinatorial Theory and its Applications II} (Proc. Colloq.
Balatonf\"{u}rer, 1969), North Holland, 1970, 793--799.


\end{thebibliography}
\end{document}